\documentclass[a4paper,10pt]{amsart}
\usepackage{amssymb,amsmath,amsfonts,amsthm}
\usepackage[dvips]{graphicx}
\usepackage{tikz-cd} 
\usepackage{psfrag}
\usepackage{xcolor}
\usepackage{textcomp}

\newtheorem{theorem}{Theorem}[section]

\newtheorem{addendum}[theorem]{Addendum}

\newtheorem{proposition}[theorem]{Proposition}
\newtheorem{lemma}[theorem]{Lemma}
\newtheorem{corollary}[theorem]{Corollary}

\newtheorem{definition}[theorem]{Definition}
\newtheorem{remark}[theorem]{Remark}

\newtheorem{question}{Question}

\newcommand{\new}[1]{{\bf#1}}

\newcommand{\occult}[1]{}

\usepackage{soul}

\newcommand\conorm{\operatorname{conorm}}

\newcommand\supp{\operatorname{supp}}

\begin{document}

\title[On the number of ergodic measures of maximal entropy] {On the number of ergodic measures of maximal entropy for partially hyperbolic diffeomorphisms with compact center leaves}


\begin{abstract}
In this paper, we study the number of ergodic measures of maximal entropy for partially hyperbolic diffeomorphisms defined on $3-$torus with compact center leaves. Assuming the existence of a periodic leaf with Morse-Smale dynamics we prove a sharp upper bound for the number of maximal measures in terms of the number of sources and sinks of Morse-Smale dynamics. A well-known class of examples for which our results apply are the so-called Kan-type diffeomorphisms admitting physical measures with intermingled basins.
\end{abstract}

\author{ Joas Elias Rocha and Ali Tahzibi}

\address{A.~Tahzibi, Instituto de Ci\^{e}ncias Matem\'aticas e de Computa\c{c}\~ao, Universidade de S\~ao Paulo (USP), \emph{E-mail address:} \tt{tahzibi@icmc.usp.br}}

\address{J.~Rocha, Unidade Acad\^emica de Cabo de Santo Agostinho, Universidade Federal Rural de Pernambuco, \emph{E-mail address:} \tt{joas.rocha@ufrpe.br}}

\thanks{A.T is supported by FAPESP 107/06463-3 and CNPq (PQ) 303025/2015-8.}

\maketitle

\section{Introduction}

Measures of maximal entropy (M.M.E) are global maxima for the Kolmogorov-Sinai entropy map $\mu \rightarrow h_{\mu}(f)$ where $f: M \rightarrow M$ is a continuous transformation defined on a compact metric space $M$. In this paper by m.m.e we mean {\bf ergodic} measures of maximal entropy. These measures, when exist, are considered among natural invariant measures of $f$.   Newhouse \cite{Newhouse89} studied an upper bound on the defect in uppersemicontinuity of topological and metric entropy and proved upper semi continuity of $\mu \rightarrow h_{\mu}(f)$ under $C^{\infty}$ regularity condition on $f.$ As a consequence any $C^{\infty}$ diffeomorphism of a compact manifold admits at least one measure of maximal entropy.

However, there are examples of (with high regularity $C^r,0 < r < \infty$) diffeomorphisms without any measure of maximal entropy (See for instance \cite{misiu}, \cite{Buzzi}.)

It is natural to study the existence of maximal entropy measures among diffeomorphisms which preserve some additional structure.  In the context of dynamical systems admitting some hyperbolicity, the question of the existence of maximal measures has been studied by several authors. After the classical results of existence and uniqueness in the uniformly hyperbolic case, a next step is to study the partially hyperbolic setting. See  for instance \cite{CPZ}, \cite{U}, \cite{BFSV}, \cite{RHRHTU} among other results in the references.

For a diffeomorphism $f:M\to M$ of a compact manifold to itself recall the norm and conorm with respect to a subspace of $V\subset T_xM$ for some $x\in M$: $\|Df| V\|:=\max\{\|Tf(v)\|:v\in V,\; \|v\|=1\}$ and
 $$\begin{aligned}
   &\conorm(Df| V) := \min\{\|Tf(v)\|:v\in V,\; \|v\|=1\}.
 \end{aligned}$$
 A splitting $E\oplus F$ is dominated\footnote{This is sometimes called \emph{pointwise} domination, see \cite{AbdenurViana}.} if it is nontrivial, invariant, and if there is some $N\geq1$ such that, for all $x\in M$:
   $$
      \|Df^N|E_x\|<\frac12\conorm(Df^N|F_x).
   $$

\begin{definition}
A diffeomorphism is  \new{(strongly) partially hyperbolic} if there is an invariant splitting of the tangent bundle: $TM=E^s\oplus E^c\oplus E^u$ such that $E^s\oplus(E^c\oplus E^u)$ and $(E^s\oplus E^c)\oplus E^u$ are dominated, $E^s$ is uniformly contracted, and $E^u$ is uniformly expanded. 
\end{definition}

 It is a well-known fact that there are foliations $\mathcal{F}^{*}$ tangent to the distributions $E^{*}$
for $*=s,u$ . The leaf of $\mathcal{F}^{*}$
containing $x$ will be denoted by $\mathcal{F}^{*}(x)$, for $*=s,u$. \par
In general it is not true that there is a foliation tangent to
$E^c$.  It can fail to be true even if $\dim E^c =1$ (\cite{RHU}) for partially hyperbolic dynamics defined on $\mathbb{T}^3$. We shall say that $f$ is {\it dynamically coherent} if there exist invariant foliations tangent to $E^{c *}=E^c \oplus E^{*}$  for $*=s,u$ (and then, to $E^c$). Along this paper all partially hyperbolic diffeomorphisms will be dynamically coherent.

We shall say that a set $X$ is {\it $*$-saturated} if it is a union
of leaves of the strong foliations $\mathcal{F}^{*}$ for $*=s$ or $u$. We also say
that $X$ is $su$-saturated if it is both $s$- and $u$-saturated. The
{\it accessibility class} $Acc(x)$ of the point $x\in M$ is the
minimal $su$-saturated set containing $x$. Note that the
accessibility classes form a partition of $M$. In case there is some
$x\in M$ whose accessibility class is $M$, then the diffeomorphism
$f$ is said to have the {\it accessibility property}. This is
equivalent to say that any two points of $M$ can be joined by a path
which is piecewise
tangent to $E^s$ or to $E^u$. \par%

If the center bundle is one dimensional or  decomposes into one dimensional sub-bundles the results in \cite{DFPV}, \cite{LVY}  show the existence of m.m.e's. However the quest for uniqueness or finiteness of ergodic measures of maximal entropy is a much more delicate problem in this setting.

In this paper we address the problem of determining the number of ergodic measures of maximal entropy of $C^2$ partially hyperbolic diffeomorphisms on $\mathbb{T}^3$ with compact center leaves.  There is another important class of partially hyperbolic diffeomorphisms on $\mathbb{T}^3$ which are isotopic to Anosov diffeomorphisms. For this class R. Ures (\cite{U}) proved uniqueness of measure of maximal entropy (See also a similar result in \cite{BFSV}).

In \cite{RHRHTU} the authors proved the following dichotomy for $C^{1+\alpha}, \alpha > 0$ dynamically coherent partially hyperbolic with accessibility property and compact center leaves: Either $f$ is ``rotation type" with a unique non hyperbolic measure of maximal entropy or there exist only finitely many hyperbolic ergodic m.m.e and there exists at least one measure with negative and another with positive center Lyapynov exponent.
More precisely they proved the following theorem.
\begin{theorem} \label{dichotomy} \cite{RHRHTU}
Let $f : M \rightarrow M$ be a $C^{1+\alpha}$ partially hyperbolic diffeomorphism of a 3-dimensional closed manifold $M$. Assume that $f$ is dynamically coherent with compact one dimensional central leaves and has the accessibility property. Then $f$ has finitely many ergodic measures of maximal entropy. There are two possibilities:
\begin{enumerate}
\item  (rotation type) $f$ has a unique entropy maximizing measure $\mu$. The center Lyapunov exponent $\lambda_c(\mu)$ vanishes and $(f, \mu)$ is isomorphic to a Bernoulli shift,

\item (generic case) $f$ has more than one ergodic entropy maximizing measure, all of which with non vanishing center Lyapunov exponent. The center Lyapunov exponent $\lambda_c(\mu)$ is nonzero and $(f, \mu)$ is a finite extension of a Bernoulli shift for any such measure $\mu.$ Some of these measures have positive central exponent and some have negative central exponent. \end{enumerate}
Moreover, the diffeomorphisms fulfilling the conditions of the second item form a $C^1-$open and $C^{\infty}-$dense subset of the dynamically coherent partially hyperbolic diffeomorphisms with compact one dimensional central leaves.
\end{theorem}

In the generic case of the above theorem (second item), the number of ergodic m.m.e's is larger than one. However still one can hope uniquenss of m.m.e fixing the sign of center Lyapunov exponent. In this paper we analyse the number of m.m.e with positive (negative) exponent assuming some condition on the dynamics of one periodic leaf.

Ures, Viana and J. Yang \cite{UVY} proved an optimal quantitative result when $M \neq \mathbb{T}^3$ is a nil-manifold: 
\begin{theorem}  \label{nilnott3}\cite{UVY}
Let $f$ be a $C^2$
 partially hyperbolic diffeomoprhism on a $3-$dimensional nilmanifold $M \neq \mathbb{T}^3$. Then 
 \begin{itemize}
 \item either $f$ has a unique maximal meaure, in which case $f$ is conjugate to rotation extension of an Anosov diffeomorphism (rotation type) and the maximal measure is supported on the whole manifold and has vanishing center exponent;
 \item or $f$ has exactly two ergodic maximal measures $\mu^+$, $\mu^-$, with positive and negative center exponents, respectively.
\end{itemize} 

 \end{theorem}

We emphasize that partially hyperbolic diffeomorphisms on $3-$nilmanifolds other than $\mathbb{T}^3$ have some  nice topological properties (\cite{HHUN}, \cite{HP14}) which are essential to the proof of the above theorem: Any such diffeomorphism is accessible, dynamically coherent and it has a unique compact, invariant, $u-
$saturated (respectively, $s-$saturated) minimal subset. 

In the setting of partially hyperbolic diffeomorphisms of $\mathbb{T}^3$ there is no ``general result" on the number of measures of maximal entropy.  Multiplying Anosov diffeomorphism on $\mathbb{T}^2$ to an appropriate one dimensional dynamics it is easy to get simples examples with any (countable)number of hyperbolic m.m.e and co-existence of hyperbolic and non-hyperbolic maximal measures. So, the interesting question is to verify the number of m.m.e's under some restriction hypothesis on the center dynamics or accessibility of the partially hyperbolic system.

In Theorem \ref{dicho} (resp. \ref{rotation}) we assume that the dynamics of a center leaf is Morse-Smale (resp. with irrational rotation number). We also give examples (Theorem \ref{kan} and section \ref{lastexample} ) to show the optimality of results.

\section{Statement of results}
Let $f: \mathbb{T}^3 \rightarrow \mathbb{T}^3$ be a dynamically coherent partially hyperbolic diffeomorphism with compact center leaves. Let $\mathcal{F}^c(x)$ be a periodic leaf with period $n.$ We say $f$ has $k-$Morse-Smale dynamics on this leaf if $f^n : \mathcal{F}^c(x) \rightarrow \mathcal{F}^c(x)$ is a Morse-Smale dynamics with $k$ sinks and $k$ sources. As an easy example of such dynamics  consider direct product $f:= A \times g$ where $A$ is an Anosov diffeomorphism on $\mathbb{T}^2$ and $g$ is an apropriate (to guarantee partial hyperbolicity)  $1-$Morse-Smale dynamics on the circle. It is clear that such an example admits exactly one ergodic measures of maximal entropy with positive center Lyapunov exponent and another ergodic m.m.e's with negative exponent.  
However, $f_k$ is not accessible and it is not topologically transitive. For non trivial and transitive example we may consider Kan-Type diffeomorphisms (See Section \ref{kantype}) . These examples have two center leaves with $1$-Morse Smale dynamics. 

 Here we prove a dichotomy for all partially hyperbolic dynamics which admit a $1-$Morse-Smale dynamics. 
 In the following theorem by an $su-$torus we mean a $2-$torus which is an accessibility class.

\begin{theorem} \label{dicho}
Let $f : \mathbb{T}^3 \rightarrow \mathbb{T}^3$ be a $C^2$ dynamically coherent partially hyperbolic diffeomorphism with compact center leaves.  Suppose that $f$ has a periodic center leaf with $1-$Morse-Smale dynamics then at least one of the following occurs:

\begin{enumerate}
\item  Either the number of ergodic maximal measures is precisely two (one with positive center expnent and one with negative center exponent) or 
\item there exists an invariant $su-$torus.
\end{enumerate}

\end{theorem}

\begin{corollary}
 If $f$ satisfies the hypothesis of the above theorem and is accessible then $f$ admits one m.m.e with positive center Lyapunov exponent and one m.m.e with negative center Lyapunov exponent.
\end{corollary}
\begin{addendum}
In the above theorem if $f$ admits a periodic center leaf with $k-$Morse-Smale dynamics then either $f$ admits at most 
$2(k + \lfloor \frac{k}{2} \rfloor )$ hyperbolic m.m.e or there exists an invariant $su-$torus.
\end{addendum} \label{aden}

\begin{remark}
Let us emphasize that even assuming accessibility, without our hypothesis on Morse-Smale periodic leaf, no general result on the number of m.m.e's is known. It is well worth to recall that for any $l \geq 1$ after a $C^{l}$- local perturbation one can change the dynamics of a periodic leaf to obtain a $k-$Morse-Smale dynamics for some $k \in \mathbb{N}$ .
\end{remark}
To prove the above theorem we consider two following disjoint cases:
\begin{itemize}
\item $f$ admits an ergodic measure of maximal entropy with zero center exponent. 
\item $f$ admits no ergodic measure of maximal entropy with vanishing center exponent.
\end{itemize}

In the first case we will show the existence of an $su-$torus. 
In the second case we prove that either there exists a unique m.m.e with negative center exponent (and a unique one with positive center exponent) or there exists an invariant $su-$torus.

Another result which is immediate consequence of Proposition \ref{u-sature} is the following:
\begin{proposition}
 If $f$ satisfies the hypothesis of Theorem \ref{dicho} and all m.m.e's are hyperbolic then there are at most two m.m.e with negative center exponents (Similarly there are at most two m.m.e with positive center exponent.)
\end{proposition}

\begin{proof}
Suppose $\mu_i, i=1,2,3$ be three m.m.e with negative center Lyapunov exponent. We know that the support of each $\mu_i$ intersects any center leaf (See for instance inside the proof of Theorem \ref{dicho}). In particular $\supp(\mu_i)$ intersects the periodic center leaf  admitting $1-$Morse-Smale dynamics. By definition of $1-$Morse-Smale dynamics there exists one sink and one source on this periodic center leaf. As the support  is closed invariant set and by Proposition \ref{u-sature} supports of $\mu_i$ are disjoint we get a contradiction.

\end{proof}

 In the sequel we give examples of such dynamics with four ergodic measures of maximal entropy (two with positive and two with negative center exponent).  We also give a topologically transitive example with three ergodic measures of maximal entropy. The construction is based on examples of non-invertible maps appeared in the work of Nu\~{n}ez-Madariaga, S. Ramirez and C. Vasquez \cite{vasco}.
 The celebrated Kan example in the annulus is an endomorphism defined on $\mathbb{S}^1 \times [0, 1]$ with two physical measures with intermingled basins. Nu\~{n}ez-Madariaga, S. Ramirez and C. Vasquez show that besides these two physical measures (Lebesgue measure on $\mathcal{S}^1$) which are measure of maximal entropy (besides being physical) there exists a third measure of maximal entropy. They describe this third measure as limit of periodic measures.
  We interpret this third measure as a twin measure in the diffeomorphism setting.

\begin{theorem} \label{kan}
There exist partially hyperbolic diffeomorphisms satisfying the hypothesis of the above theorem with 4 m.m.e's two of them with negative and two others with positive center exponent. There is also topologically transitive example with 2 m.m.e with negative center exponent and one with positive center exponent.
\end{theorem}

Observe that although the second example in the above theorem is topologically transitive, still we are not able to find an accessible example:
\begin{question}
Is there any accessible partially hyperbolic diffeomorphism satisfying hypothesis of Theorem \ref{dicho} with more than one ergodic m.m.e with positive center exponent?
\end{question}
We still do not know a clear dichotomy on the hyperbolicity of m.m.e's under the hypothesis of Theorem \ref{dicho}.  
\begin{question}
Is there any example of partially hyperbolic diffeomorphism satisfying hypothesis of Theorem \ref{dicho} admitting both hyperbolic and non-hyperbolic measures of maximal entropy?
\end{question}
We  emphasize that in the Kan tye examples stated in Theorem \ref{kan} the co-existence of hyperbolic and non-hyperbolic m.m.e's is ruled out using the intermingled property of hyperbolic m.m.e's and invariance principle argument. 

In the next theorem we assume the existence of a periodic leaf with irrational rotation number dynamics and prove a dichotomy.  
\begin{theorem} \label{rotation}
Let $f: \mathbb{T}^3 \rightarrow \mathbb{T}^3$ be a $C^2$-partially hyperbolic diffeomorphism, dynamically coherent with compact central leaves with one periodic leaf with irrational rotation number then:
\begin{enumerate}
\item Either there is unique measure of maximal entropy $\mu$, moreover $\lambda^c(\mu)=0$ and $f$ is conjugate to rotation extension of Anosov homeomorphism or
\item $f$ admits exactly two ergodic measures of maximal entropy. Both  measures are hyperbolic and have opposite sign of center Lyapunov exponent.
\end{enumerate}
\end{theorem}

Observe that in the above theorem we are not assuming accessibility of $f$. Instead we are assuming existence of a periodic leaf with irrational rotation dynamics. We also give examples (subsection \ref{lastexample}) for each item of the  dichotomy announced in the above theorem.



\section{Quotient Dynamics and Conditional measures}

Let $f : \mathbb{T}^3 \rightarrow \mathbb{T}^3$ be a $C^1-$ partially hyperbolic diffeomorphism with compact center leaves. We define the quotient space $M_c$ and quotien dynamics $f_c$ by means of the natural projection $\pi : \mathbb{T}^3 \rightarrow M_c := \mathbb{T}^3/\mathcal{F}^c$ which sends every point in $\mathcal{F}^c(x)$ to $[x].$ It is known that $M_c$ is homeomorphim to $\mathbb{T}^2$ and $f_c$ is an Anosov homeomorphism (See Theorem 3 in \cite{RHRHTU}). So it is topologically conjugate to an Anosov diffeomorphism on $\mathbb{T}^2.$ 

From now on we fix some notations: $\mathcal{F}^s, \mathcal{F}^c$ and $\mathcal{F}^u$ are respectively stable, center and the unstable foliation invariant by $f.$ 
The points in $M_c$ are denoted by $[.]$. For any point $[x] \in M_c$ by $\mathcal{W}^s ([x])$ and $\mathcal{W}^u([x])$ we mean the stable and unstable set of $[x].$  As $f$ is conjugate to an Anosov diffeomorphism on $\mathbb{T}^2$ it comes out that both (topological) foliations $\mathcal{W}^u$ and $\mathcal{W}^s$ are minimal.  Observe that 
\begin{itemize}
\item $\pi (\mathcal{F}^{cs}) (x) = \mathcal{W}^s(\pi(x)),$
\item $\pi (\mathcal{F}^{cu}) (x) = \mathcal{W}^u(\pi(x)).$
\end{itemize}
 
Moreover, $f_c$ has a unique measure of maximal entropy $\eta$ which is fully supported on $M_c.$ This measure  has local product structure:
For any $[x] \in M_c$ there exists an open set $U$ around $[x],$ a homeomorphism $$\alpha: \mathcal{W}^s_{loc} ([x]) \times \mathcal{W}^u_{loc} ([x]) \rightarrow U$$  and measures $\eta^u_{[x]}, \eta^s_{[x]}$ supported on respectively $\mathcal{W}^s_{loc} ([x])$ and  $ \mathcal{W}^u_{loc} ([x])$ such that:
$$
 \eta|_{U} \sim \alpha_{*}(\eta^u_{[x]} \times \eta^s_{[x]}),
$$
where $\sim$
 denotes equivalence of measures.

Take any invariant measure $\mu$ for $f$ and let $\nu = \mu\circ \pi^{-1}$. By the Ledrappier-Walters variational principle \cite{LW}
$$
 \sup_{\hat{\mu} :\hat{ \mu} \circ \pi^{-1} = \nu } h_{\hat{\mu}}(f) = h_{\nu} (f_c) +
 \int_{M_c} h(f, \pi^{-1}(y)) d\nu(y).
$$
 Since $\pi^{-1} (y),\, y \in M_{c}$, is a circle and its iterates have bounded length we have that  $h_{top}(f, \pi^{-1}(y)) =0,$  that is,  fibers does not contribute to the entropy.  Hence, by the above equality and the well-known fact that $h_{\mu} (f) \geq h_{\nu}(f_{c})$ we conclude that $h_{\mu}(f) = h_{\nu} (f_{c}).$  Using the usual variational principle this implies that the topological entropies of $f$ and $f_{c}$ coincide.
 In particular,
 the set of entropy maximizing measures of $f$ coincides with the subset of ergodic measures which projects down to $\eta$, the entropy maximizing measure of $f_{c}.$
 \subsection{ Invariance principle and support of m.m.e}
In the setting we are working, as center foliation is given by compact leaves and the quotient space is homeomorphic to $\mathbb{T}^2$ which is a  separable metric space  we can apply Rokhlin disintegration theorem   and conclude that for any probability $\mu$ there is a unique family of conditional measures $\mu^c_{[x]}$ (probability supported on $\mathcal{F}^c(x)$) defined $\eta:=\pi_{*} \mu-$almost every where and $$\mu = \int_{M_c} \mu^{c}_{[x]} d \eta ([x]).$$

Sometimes, it is more conveniente to use the notation $\{\mu^c_x\}_{x \in M}$ for conditional measures along center foliation. Using this notation we have $\mu^c_x = \mu^c_{y}$ for $y \in \mathcal{F}^c (x)$ and $\mu^c_x$'s are probability measures where 
$$
 \mu = \int_M \mu^c_x d \mu(x).
$$

As we mentioned before, if $\mu$ is any measure of maximal entropy then $\nu = \pi_* \mu$ is maximal entropy measure for the quotient dynamics which is Anosov. So $\nu$ has local product structure and one may apply the following invariance principle due to Avila-Viana. Recall that by center Lyapunov exponent $\lambda_c(x)$ we mean the following limit:
$$
 \lim_{n \rightarrow \infty} \frac{1}{n} \log \|Df^n|_{E^c(x)}\|.
$$
By Oseledets theorem, given any invariant probability measure the above limit exists for almost every point $x$.  
 
Now we announce an Invariance principle due to Avila-Viana  \cite{AV}. 
 Let us concentrate on a class of partially hyperbolic dynamics which include the systems under consideration in this paper.
Let $f: M \rightarrow M$ be a partially hyperbolic dynamics satisfying the following conditions:
\begin{itemize}
\item H1. $f$ is dynamically coherent with all center leaves compact,
\item H2. $f$ admits global holonomies, that is, for any $y \in \mathcal{F}^u(x)$ the holonomy map $H^u_{x,y}: \mathcal{F}^c(x) \rightarrow \mathcal{F}^c(y)$ is a homeomorphism. For any $z \in \mathcal{F}^c(x)$, $H^u_{x,y} (z) = \mathcal{F}^u(z) \cap \mathcal{F}^c(y).$
\item H3. $f_c$ is a transitive topological Anosov homeomorphism, where $f_c$ is the induced dynamics satisfying $f_c \circ \pi = \pi \circ f$ and $\pi : M \rightarrow M/\mathcal{F}^c$ is the natural projection to the space of central leaves. In particular there are two foliations $\mathcal{W}^s$ and $\mathcal{W}^u$ which are stable and unstable sets for $f_c.$
\end{itemize}

\begin{theorem}[\cite{AV}] (Invariance Principle)\label{invp}
Let $f:M\rightarrow M$ be a partially hyperbolic diffeomorphism with one-dimensional compact center leaves satisfying $H1, H2$ and $H_3.$  Let
$\mu$ be an $f$-invariant probability measures whose projection
$\nu =\pi_* \mu$
is probability measure that has local product structure. Assume that $\lambda_c (x)=0$ for $\mu-$almost every point. Then $\mu$ admits a disintegration
$\{\mu_{[x]}: \, [x] \in  M_c\}$ which is $s$-invariant and $u$-invariant and whose conditional
probabilities $\mu_{[x]}$ vary continuously with $[x]$
on the support of $\nu$.
\end{theorem}

	
	

The following lemma is an immediate corollary of the uniqueness of Rokhlin disintegration theorem and above invariance principle result.
\begin{lemma}  \label{inv-MS}
Let $f:\mathbb{T}^3\longrightarrow\mathbb{T}^3$ be $C^2$ partially hyperbolic diffeomorphism, dynamically coherent with compact central leaves and suppose that $f$ has a maximal entropy ergodic measure $\mu$ with zero exponent. If  $\{\mu_{[x]}\}_{[x]\in\mathbb{T}^3/\mathcal{F}^c}$ is a continuous disintegration of $\mu$ then, {\bf for all} $[x]$:
	
	$$f_*\mu_{[x]}=\mu_{f_c([x])}$$
	
\end{lemma}
\begin{proof}
Observe that by invariance of $\mu$ and uniqueness of disintegration for almost every $x$ we have $f_* \mu_x = \mu_{f(x)}.$ As $[x] \rightarrow \mu_{[x]}$ is continuous on $\supp(\pi_{*}\mu) = \mathbb{T}^3/\mathcal{F}^c,$ we get the desired result.

\end{proof}
\subsection{Support of maximal entropy measures}
Let us remind some basic properties of support of measures of maximal entropy.

Ures-Viana-Yang  proved two crucial facts which hold for all systems satisfying hypothesis of Theorem \ref{nilnott3}. 
 
 Although the main theorem in \cite{UVY} is formulated for $M \neq \mathbb{T}^3$, the following proposition holds in the case of partially hyperbolic $f: \mathbb{T}^3 \rightarrow \mathbb{T}^3$ with one dimensional compact center leaves which is the setting we are interested here.  
The proof is exactly the same as in \cite{UVY}. 
 \begin{proposition} \label{u-sature} (\cite{UVY} Lemma 4.4 and Proposition 3.12) The supports of two ergodic m.m.e with negative (positive) exponent are disjoint in the absence of m.m.e with zero central exponent. Moreover the support of any m.m.e with non-positive (resp. non-negative) center exponent is saturated by $\mathcal{F}^u$ (resp. $\mathcal{F}^s$) leaves.
 \end{proposition}

We just observe that in the case of $M \neq \mathbb{T}^3$, as any partially hyperbolic diffeomorphism is accessible, by Theorem \ref{dichotomy} if there exists any hyperbolic m.m.e, then there does not exist any non-hyperbolic m.m.e. As a consequence, for $M \neq \mathbb{T}^3,$ the  disjointness of supports in the above proposition does not need the hypothesis of absence of non-hyperbolic m.m.e  in \cite{UVY}.

\begin{remark} \label{supports}
Observe that in the case of $M \neq \mathbb{T}^3$ every partially hyperbolic diffeomorphism is accessible and the support of  m.m.e with zero center exponent is the whole $M$ (if it exists). In our setting $M = \mathbb{T}^3$ by above proposition one conclude that the support of any m.m.e with vanishing center exponent is saturated by accessibility classes.
\end{remark}

 The assumption  $M $ being nil-manifold different from $\mathbb{T}^3$ in \cite{UVY} is crucial to conclude that $f$ is dynamically coherent and there exists unique compact invariant, $u-$saturated (respectively $s-$saturated) minimal subset (Proposition 1.9 and 6.4 in \cite{HP14}). So by the above proposition  the authors can prove that there exist at most one m.m.e with negative (resp. positive) center exponent.

\subsection{Twin measure construction}

In this subsection we recall a method to obtain new measures of maximal entropy beginning from hyperbolic m.m.e's which initially appeared in \cite{RHRHTU}  and  denoted by twin measure construction.

Let $f$ be as in Theorem \ref{dichotomy} (without  assuming accessibility property). Suppose $\mu$ is an hyperbolic measure of maximal entropy. Then using an argument as in Ruelle-Wilkinson one concludes that conditional measures of $\mu$ along center leaves are atomic. As $\mu$ is hyperbolic, the center Lyapunov exponent is non-zero and we suppose that it is negative. Then for typical (for measure $\mu$) $x$ the intersection of Pesin stable manifold with center leaf is an open curve $\mathcal{W}^{ws}(x)$ (weak stable manifold) inside $\mathcal{F}^c(x).$ So there exists $Z$ with full $\mu$ measure where points in $Z$ admit Pesin stable manifold. This enables us to define an injective measurable map $\beta: Z \rightarrow \mathbb{T}^3$ by $\beta(x):=  \partial^{+} \mathcal{W}^{ws}(x)$ where $\partial^{+} \mathcal{W}^{ws}(x)$ stands for the boudary point of $ \mathcal{W}^{ws}(x)$ going from $x$ to positive orientation. As we are assuming that $f$ preserves the orientation of $\mathcal{F}^c$ we have that $\beta \circ f = f \circ \beta$ and $\beta$ is an isomorphism between $(f, \mu)$ and $(f, \beta_{*} \mu).$ We call $\beta_{*} \mu$ as a twin measure of $\mu.$ Clearly there may be constructed another twin measure just substituing $\partial^{+} \mathcal{W}^{ws}(x)$ by $\partial^{-} \mathcal{W}^{ws}(x).$ 

It is clear that $\beta_* \mu$ and $\mu$ has the same entropy and so if $\mu$ is m.m.e, its twin is too.

\section{Proof of results}

\subsection{Proof of Theorem \ref{dicho}} Let $\mathcal{F}^c(a)$ be a periodic leaf through $a$ where $f^l \mathcal{F}^c(a)$ is a $1-$Morse-Smale dynamic with attractor $a$ and repeller $r.$ Without loss of generality we assume that $a$ is a fixed point and so $\mathcal{F}^c(a)$ is a fixed leaf.
We consider two following disjoint cases:
\begin{itemize}
\item $f$ admits an ergodic measure of maximal entropy with zero center exponent. 
\item $f$ admits no ergodic measure of maximal entropy with vanishing center exponent.
\end{itemize}

In the first case we will show the existence of an $su-$torus. 
In the second case we prove that either there exists a unique m.m.e with negative center exponent (and a unique one with positive center exponent) or there exists an invariant $su-$torus.

Let us deal with first case:
 let $m$ be maximal measure with zero exponent and $\{m_{[x]}\}_{[x] \in \mathbb{T}^3/\mathcal{F}^c}$ be the disintegration along center leaves. By Lemma \ref{inv-MS} we have $m_{\pi(a)} = c_1 \delta_a + c_2 \delta_r$. Suppose that $c_1 \neq 0.$ As the conditional measures $m_{[x]}$ are invariant under stable and unstable holonomies we have:
$$
\{ h_{\gamma} (a) \cap \mathcal{F}^c(a): \gamma \quad \text{is} \quad su-\text{path}\} \subseteq \supp{m_{\pi(a)}} \subseteq \{a, r\}.
$$

Indeed, 
As $h_{\gamma}$ is a homeomorphism for any $su-$path $\gamma$ and the quotient dynamic is transitive Anosov homeomorphism and accessible, we conclude that:
$$
1 \leq  |\{h_{\gamma}(a) \cap \mathcal{F}^c(p): \gamma \quad su-\text{path} \}| \leq 2,
$$
for any $p \in \mathbb{T}^3$ and any $su-$path $\gamma.$ In other words,   $$Card( Acc(a) \cap \mathcal{F}^c(p)) = Card( Acc(a) \cap \mathcal{F}^c(a)) \leq 2 $$
for any $p \in \mathbb{T}^3.$ 
Using this we can show that $Acc(a)$ is compact. Then 
$$ \pi|_{Acc(a)} : Acc(a) \rightarrow \mathbb{T}^2$$
is a covering map with one or two sheets. As $Acc(a)$ is compact and connected, it is a torus.

Now we deal with the second case: 
First of all, we know that $f$ at least admits an m.m.e with negative center exponent. Suppose that there are two ergodic measures of maximal entropy $\mu_1, \mu_2$ with negative center exponent (the argument for two measures with positive exponent is similar). Observe that for any maximal measure $\eta$, $\supp(\eta) \cap \mathcal{F}^c(x) \neq \emptyset$ for all $x \in \mathbb{T}^3.$ Indeed, as $\pi_{*} \eta$ is the measure of maximal entropy for a topologically transitive Anosov homeomorphism, it is fully supported on $\mathbb{T}^3/ \mathcal{F}^c$ and by definition $\pi (\supp(\eta)) = \supp(\pi_*(\eta))$ and this implies $\supp(\eta) \cap \mathcal{F}^c(x) \neq \emptyset.$ 

\begin{lemma} \label{geral}
There exists $i \in \{1, 2\}$ such that $\supp(\mu_i)$   intersects each center leaf in exactly one point.
\end{lemma}

\begin{proof}
 There is a neighbourhood $U \subset \mathbb{T}^2$ around $\pi(r)$ such that for any $x \in U$ by local product structure $\mathcal{W}^u_{loc}(x) \cap \mathcal{W}^s_{loc} (\pi(r)) \neq \emptyset.$We prove the lemma  by contradiction. As $\pi_* (\mu_1) = \pi_*(\mu_2)$ we may suppose that there exists $[p] \in U$ with 
$Card (\pi^{-1}([p]) \cap \supp(\eta) \geq 2$ for $\eta \in \{\mu_1, \mu_2\}.$ 
Let $p_1, q_1 \in \pi^{-1}([p]) \cap \supp(\mu_1)$ and $p_2, q_2 \in \pi^{-1}([p]) \cap \supp(\mu_2).$ 

Consider the local unstable holonomy:
$$
h^u: \pi^{-1}([p]) \rightarrow \pi^{-1}(\mathcal{W}^u_{loc}([p]) \cap \mathcal{W}^s_{loc}(\pi(r))).
$$

By $u-$saturation of support of m.m.e with negative center exponent (Proposition \ref{u-sature}), we have $\hat{p}_i:= h^u(p_i), \hat{q}_i := h^u(q_i) \in \supp(\mu_i).$ 

Now consider local stable holonomy:

$$
h^s:  \pi^{-1}(\mathcal{W}^u_{loc}([p]) \cap \mathcal{W}^s_{loc}(\pi(r)))  \rightarrow \mathcal{F}^c(r).
$$

As $\mathcal{F}^c(r)$ is one dimensional with $1-$Morse-Smale dynamics for each $i \in \{1,2\}$ we have 
\begin{itemize} 
\item $d(f^n(h^s (\hat{p_i}), a) \rightarrow 0$ or
\item $d(f^n(h^s (\hat{q_i}), a) \rightarrow 0.$
\end{itemize}
 As $\hat{p_i}, \hat{q_i} \in \supp{\mu_i}$ and support of a measure is closed invariant subset then  $a \in \supp(\mu_i)$ for both $i=1, 2$ and this gives a contradiction as $\supp(\mu_1) \cap \supp(\mu_2) = \emptyset$ again using Proposition \ref{u-sature} (Observe that we are in the case where no m.m.e has zero center Lyapunov exponent).

\end{proof}

By the above lemma we conclude that there exists $i \in \{1,2\}$ such that  $$Card (\supp(\mu_i) \cap \{a, r\}) = 1.$$ and for simplicity we denote it by $\mu$ and assume that $r = \supp(\mu) \cap \mathcal{F}^c(r).$

\begin{lemma} 
$\supp(\mu)$ contains both stable and unstable leaf of $r.$
\end{lemma}

\begin{proof}
The inclusion of $\mathcal{F}^u(x) \subset \supp(\mu)$ is by $u-$saturtion of support. For the inclusion of the stable leaf, observe that for any point $z \in \mathcal{F}^{cs}(r)$ we have $\supp(\mu) \cap \mathcal{F}^c(z)$ is a unique point. We claim that this unique point should be $h^s(r) =  \mathcal{F}^s(r) \cap \mathcal{F}^c(z)$. Indeed, if not $y \in \supp(\mu) \cap \mathcal{F}^c(z)$  with $\lim_{k \rightarrow \infty} d(f^k(y), a) = 0.$ This implies that $a \in \supp(\mu)$ which is a contradiction.
\end{proof}

To finish the proof of the theorem let $S : \mathbb{T}^3 / 
\mathcal{F}^c \rightarrow \mathbb{T}^3$ be the section which associates $$[p] \rightarrow S([p]) := \supp(\mu) \cap \pi^{-1}([p])$$ for any $[p] \in  \mathbb{T}^3 / 
\mathcal{F}^c.$

As $\supp(\mu)$ is $u-$saturated we have that $S(\mathcal{W}^u([p])) = \mathcal{F}^u(S([p])).$ We have also proved that $\mathcal{F}^s(r) \subset \supp(\mu)$ which yields
 \begin{equation} \label{atomess}
 S(\mathcal{W}^s(\pi(r))) = \mathcal{F}^s(r). 
 \end{equation}

Finally we prove that $\supp(\mu)$ is an $su-$saturated set. Let $q \in \supp(\mu).$ We need to prove that $\mathcal{F}^s(q) \in \supp(\mu).$ We claim that there is a sequence $p_n$ with $\mathcal{F}^s(p_n) \in \supp(\mu).$ As support is a closed set and stable foliation is continuous we conclude that $\mathcal{F}^s(q) \in \supp(\mu).$ Indeed, 
note that $\mathcal{W}^s(\pi(r)) \cap \mathcal{W}^u(\pi(q))$ accumulates on $\pi(q).$  Let $[z_n] \in \mathcal{W}^s(\pi(r)) \cap \mathcal{W}^u(\pi(q))$ converging to $\pi(q)$. Observe that as $\mu$ has negative center exponent 
\begin{equation} \label{atomesu}
S(\mathcal{W}^u(\pi(q))) = \mathcal{F}^u(q).
\end{equation}

Using (\ref{atomess}), (\ref{atomesu}) we have:
\begin{equation} p_n:=S([z_n]) = \mathcal{F}^s(r) \cap \mathcal{F}^u(q) \cap \pi^{-1}([z_n]).
\end{equation}
This concludes the proof of the claim and finishes the proof of Theorem \ref{dicho} in the case where $f$ admits $1-$Morse-Smale dynamics on a periodic center leaf.

The proof of Addendum \ref{aden} is similar to the above theorem. In fact if there are $k + \lfloor \frac{k}{2} \rfloor + 1$ m.m.e with negative center exponent using similar argument as in Lemma \ref{geral} we get at least one  $\mu_i$ such that $\supp(\mu_i)$ intersects each center leaf at just one point and repeating the above arguments we get an $su-$torus.

\subsection{Proof of theorem \ref{rotation}}

\begin{proof} 
	
 We consider following disjoint cases:
\begin{itemize}
	\item $f$ admits an ergodic measure of maximal entropy with zero center exponent. 
	\item $f$ admits no ergodic measure of maximal entropy with vanishing center exponent.
\end{itemize}	
	
First, Let us suppose that $f$ has an ergodic maximal entropy  measure $m$ with zero center Lyapunov exponent. Then, by the invariance principle Theorem \ref{invp}, there exists  continuous disintegration $\{m_x^c\}$ along the central foliation which is invariant be stable and unstable holonomies.

The disintegration $\{m_x^c\}_x$ is invariant and $f^\tau|_{\mathcal{F}^c(a)}\circ h=h\circ R_\alpha$ for $\alpha\in\mathbb{R}\setminus\mathbb{Q}$. Since $R_\alpha$ is uniquely ergodic, then  $h_*Leb_{\mathbb{S}}$ is the unique maximal entropy measure for $f^\tau|_{\mathcal{F}^c(a)}$. Moreover, this meausure has no atoms and satisfies $\supp(m_a)=\supp h_*Leb_{\mathbb{S}}=\mathcal{F}^c(a)$. Using acessibility of the quotient map $f_c:\mathbb{T}^3/\mathcal{F}^c\longrightarrow \mathbb{T}^3/\mathcal{F}^c$, we conclude that $m_x^c$ has no atoms and is fully supported on center leaf for every $x\in\mathbb{T}^3$.

Then, we define an action of $\mathbb{S}$ on $M$ that commutes with $f$. Take an orientation for $\mathcal{F}^c$. Now, identifying $\mathbb{S}$ with $[0,1]/\mbox{mod1}$, define 
$\Gamma:\mathbb{S}\times M\longrightarrow M$ in such a way that $m_x([\Gamma(\theta,x),x])=\theta$. This action is well defined and continuous because $\supp(m_x)=\mathcal{F}^c(x)$, the conditional measures has no atoms and the disintegration is continuous. Thus, by the invariance of the disintegration, we conclude that $f(\Gamma(\theta,x))=\Gamma(\theta,f(x))$. 

Similar to what have been done in \cite{RHRHTU} we conclude that $f$ is conjugate to $\bar{f}$ where $\bar{f}$ is a rigid rotation extension of  Anosov homeomorphism $f_c$. 
As $f$ has a periodic leaf (with period $\tau$ with irrational rotation number dynamics, then there exists a periodic leaf $\mathcal{F}^c(a)$ where $\bar{f}^{\tau}$ is an irrational rotation.

For the uniqueness, it is enough to show that $\bar{f}$ has a unique measure of maximal entropy $m$ such that the conditional measures $m_x$ along center leaves are Lebesgue and the quotient measure is the measure of maximal entropy for $f_c$. Let us suppose that $\eta$  is a maximal entropy ergodic measure. Then, $\lambda^c(\eta)=0$ because $\bar{f}$ acts isometrically on center leaves. If $\{\eta_x\}_x$ is a disintegration of $\eta$, then $\eta_a$ is the unique that satisfies $\eta_a=f^\tau_*\eta_a$ and by consequence, $\eta_a$ is Lebesgue. Since the quotient map $f_c$ is accessible and the disintegration of $\eta$ is invariant by holonomies, we get $\eta_x=m_x$ for every $x\in\mathbb{T}^3$. This concludes the first part.

Now, let us suppose that there is no maximal entropy ergodic measure with zero central expoent and let $\mu$ be an m.m.e with $\lambda_c(\mu)<0$. We will show that $\mu$ is the unique m.m.e with negative central exponent. If  $\eta\neq\mu$, is another m.m.e with $\lambda_c (\eta)<0$, we must have $\supp(\eta)\cap\supp(\mu)=\emptyset$. To get uniqueness is enough to show that $\supp(\mu)=\mathbb{T}^3$.

 Let be $\mathcal{F}^c(p)$ the periodic leaf with irrational dynamics. As observed before, we have $\supp(\mu)\cap \mathcal{F}^c(x)\neq\emptyset$ for all $x\in\mathbb{T}^3$, in particular, $\supp(\mu)\cap \mathcal{F}^c(p)\neq \emptyset$.  By consequence,  $\mathcal{F}^c(p)\subset \supp(\mu)$ since $\mathcal{F}^c(p)$ has irrational dynamics and $\supp(\mu)$ is invariant. Once the $\supp(\mu)$ is saturated by unstable leaves,  we have $\mathcal{F}^{cu}(p)\subset \supp(\mu)$. Besides, $\overline{\mathcal{F}^{cu}(p)}=\mathbb{T}^3$  which gives $\supp(\mu)=\mathbb{T}^3$ and concludes the proof.
 
 A twin argument shows that there exists also a unique measure of maximal entropy with positive exponent.

\end{proof}

\subsection{Proof of Theorem \ref{kan}, Kan type examples} \label{kantype}
The Kan  example \cite{kan} is a map of the cylinder $\mathbb{A}= \mathbb{S}^1 \times [0,1] \rightarrow \mathbb{A}$, defined as:

$$
f(x, t) = (3x  \, \text{mod-} \mathbb{Z}, t+ \frac{t(1-t)}{32} cos(2 \pi x)), \quad \forall (x, t) \in \mathbb{S}^1 \times  [0,1].
$$
This map is well known because of its intricate proeprty of admitting two physical measures with intermingled basins.
Recall that for the basin $B(\mu)$ of an invariant measure $\mu$  is the set of points $x$ such that $$\lim_{n \rightarrow \infty} \frac{1}{n} \sum_{j=0}^{n-1} \delta_{f^{j}(x)} = \mu.$$
An invariant measure is physical if $Leb(B(\mu)) > 0.$ Two physical measures are called intermingled if for every open set $U$ we have $Leb(B(\mu_1) \cap U) > 0$ and $Leb(B(\mu_2) \cap U) > 0.$

In the above Kan's example, the (one dimensional) Lebesgue measures $\mu_1, \mu_2$ on the boundary circles are invariant measures  and Kan proved that $\mu_1$ and $\mu_2$ are physical measures with  $Leb(B(\mu_1) \cup B(\mu_2)) =1$ and more strikely, the basins are intermingled. 

Similarly one may define an skew product on $\mathbb{T}^2 \times [0,1]$ as follows:
$$
F(x, t):= (A(x), f_x(t)), 
$$ where $A:  \mathbb{T}^2 \rightarrow \mathbb{T}^2 $ is a linear Anosov diffeomorphism with eigenvalues $\mu < 1 < \lambda$. For each $x \in \mathbb{T}^2$ the function $f_x: [0, 1] \rightarrow [0,1]$ is a diffeomorphism fixing the boundary of $[0, 1].$ Take two fixed points of $A$ called $p, q.$ We require that $f_p$ and $f_q$ have exactly two fixed points each, a source at $t=1$ (respectively $t=0$) and a sink at $t=0$ (respectively $t=1$). Furthermore, $\mu < |f_x^{'} (t)| < \lambda$ and 
$$
 \int \log f_x^{'}(0) dx < 0 \quad \text{and} \quad  \int \log f_x^{'}(1) dx < 0.
$$ 
Under these conditions $F$ has two intermingled SRB measures which are normalized Lebesgue measure of each boundary torus. Under some more conditions $F$ is also transitive (see \cite{BDV}.) Observe that as $f_x, x \in \mathbb{T}^2$ are orientation preserving if $(x, t) \in B(\mu_1)$ then for all $0 \leq s \leq t$ we have $(x, s) \in B(\mu_1)$. Similarly if $(x, t) \in B(\mu_2)$ then $(x, s) \in B(\mu_2)$ for all $1 \geq s \geq t$. As $B(\mu_1) \cup B(\mu_2)$ has full Lebesgue measure we conclude that for Lebesgue almost every $x \in \mathbb{T}^2$ there exists $\sigma (x) \in (0, 1)$ where 
\begin{equation} \label{sigma(x)}
(x, s) \in B(\mu_1) \quad \text{if} \quad  s < \sigma (x) \, \text{and}\,
(x, s) \in B(\mu_2) \quad \text{if}  \quad  s > \sigma(x).
\end{equation}

To get a diffeomorphism on a boundaryless manifold,
we consider two such examples and glue them  to find a partially hyperbolic diffeomorphism of $\mathbb{T}^3$ admitting two SRB measures (See also \cite{GC} and \cite{VU}). 
 
Take $\tilde{G} : \mathbb{T}^2 \times [0,1] \rightarrow \mathbb{T}^2 \times [0,1] $ as follows:
\begin{equation} \label{doublekan}
\tilde{G}(x, t)= \left\{
\begin{array}{ll}
 (A(x), 1 - \frac{1}{2} f_x(2t)) \qquad &  0 \leq t \leq \frac{1}{2}\\
 (A(x), \frac{1}{2}f_x(2(1-t)) )  &   \frac{1}{2} \leq t \leq 1.
\end{array} \right.
\end{equation}
Clearly the above map defines a diffeomorphism on $G: \mathbb{T}^3 \rightarrow \mathbb{T}^3$ just gluing the boundaries.

Although $G$ is topologically transitive, as $G^2$ is not transitive we conclude that $G$ is not topologically mixing. So it is still a question whether one can construct topologically mixing partially hyperbolic diffeomorphism with intermingled basin physical measures.
However, see the result of Gan and Shi \cite{GS} where they prove $C^2$ robust topological mixing of Kan example on the annulus. We observe also that Bonatti-Potrie \cite{BPo} had constructed examples of mixing diffeomorphisms with intermingles basins physical measures, however their examples are not strongly partially hyperbolic. In fact in their examples $T \mathbb{T}^3= E^{cs} \oplus E^u$ is the invariant tangent bundle splitting.

Observe that the two torus $\{t=0\}, \{t=1/2\}$ are invariant and support the SRB measures with intermingled basins on the $2-$torus. We mention that Ures and Vasquez \cite{VU} proved  the existence of $su-$torus for any partially hyperbolic $C^r, r > 1$ diffeomorphism in $\mathbb{T}^{3}$ with intermingled physical measures. In particular such diffeomorphisms are not accessible and so the set of these diffeomorphisms has empty interior in $C^r$ topology.

The Lebesgue measures supported on these tori $\mu_1, \mu_2$ besides being physical, are also ergodic measures of maximal entropy for $F.$ (and $G$) Both $\mu_1$ and $\mu_2$ have negative center Lyapunov exponent. Observe that $\mu_1$ and $\mu_2$ as measures defined on $\mathbb{T}^3$ have atomic disintegration along the center foliation.

Besides these two measures $F$ (also $G$) admits another (unique) measure of maximal entropy $\nu$ with positive center exponent. In fact $\nu$ is twin measure of both $\mu_1$ and $\mu_2.$

Indeed, $\nu_1 : = \displaystyle{\int_{\mathbb{T}^2} \delta_{\sigma(x)} d Leb(x)}$ is an ergodic $F-$invariant measure such that $\pi_{*} \nu_1 = Leb$ which implies that $\nu_1$ is a maximal entropy measure. Observe that $\sigma(x)$ is the boundary point of the weak (Pesin)stable set of $(x, 0)$ which coincides with the boundary of weak (Pesin) stable set  of $(x, 1)$ for Lebesgue almost every $x.$ So, the center Lyapunov exponent of $\nu_1$ is non-negative.

We claim that $\lambda^c (\nu) \neq 0.$ If it is not the case, then by Invariance principle the disintegration of $\nu$ is defined continuously and this implies that $x \rightarrow \sigma(x)$ can be extended continuously to whole $\mathbb{T}^2$ and this contradicts the fact that the basins of $\mu_1$ and $\mu_2$ are intermingled.  

By construction of $G$ we conclude that $\nu = \frac{1}{2} (\nu_1 + G_{*} \nu_1)$ is ergodic and maximal entropy for $G.$

Finally to construct an example with four m.m.e we just consider $H= G^2$ where $\mu_1, \mu_2$ (resp. $\nu_1, \nu_2$) are respectively two measures of negative (resp. positive) center Lyapunov exponent.

 \subsection{Examples for each item of theorem \ref{rotation}} \label{lastexample}

For item 1, it is enough to consider a Anosov map  $A:\mathbb{T}^2\rightarrow\mathbb{T}^2$, a irrational rotation  $R:\mathbb{S}^1\rightarrow \mathbb{S}^1$  and then, take the product $A\times R: \mathbb{T}^3\rightarrow\mathbb{T}^3$.

 In order to give example for the second item of the dichotomy, note that if $f$ has a periodic leaf with Morse Smale dynamics, then, it cannot be conjugate to rotation extension.
 Thus, it is enough to take a Anosov $A:\mathbb{T}^2\rightarrow\mathbb{T}^2$ with two fixed points  $p$ and $q$ and a family $f_x:\mathbb{S}\rightarrow\mathbb{S}$ of   $C^2$  maps in a way that $f_p$ is a irrational map and $f_q$ is a Morse-Smale map.

In the theorem \ref{rotation}, we do not have the hypothesis of accessibility. One can ask whether there exists any example of  accessible system that satisfies the hypothesis of \ref{rotation} and still has two maximal entropy ergodic measures. Here we give a positive answer to this question.

The idea of the construction is to perturb a rotation extension in a neighborhood of a center leaf using a bump- function. If necessary, we rotate the system afterwards.
  
 Consider an accessible rotation extension $F:\mathbb{T}^2\times\mathbb{S}^1\longrightarrow \mathbb{T}^2\times\mathbb{S}^1$ (See \cite{NT} for such examples). The property of accessibility is open in $C^1$ topology (See \cite{didi}).

By definition  $F=A\ltimes R_x$, where  $R_x$ is a rotation, i.e $F(x, \theta) = (A(x), R_x (\theta))$ We can assume that $R_0=Id_{\mathbb{S}^1}$, if  not,  consider $G=(Id_{\mathbb{T}^2},R_0^{-1})\circ F$. Once the action of rotation preserves the unstable and stable foliation, the new system $G$ is still accessible 

For $t\in\mathbb{R}$, we define $g_t:\mathbb{S}^1\rightarrow \mathbb{S}^1$ by $g_{t}(\xi):=\xi+t \sin(4\pi \xi)$. For small $t$ , $g$ is a  2-Morse-Smale $C^1$ diffeomorphism close to the identity.  

Consider $R>0$ such that, there is a periodic point $p\in\mathbb{T}^2$ with orbit  outside $B(0,R)$. Now, consider a bump-function  $\eta:\mathbb{T}^2\longrightarrow [0,1]$ with $\eta(0)=1$ and $\eta(x)=0$ if $|x|>R$.

Define $H=A\ltimes(g_{l(x)}\circ R_x)$, where $l(x)=t_0\eta(x)$ for some fixed $t_0>0$.

As $F$ is stably accessible in $C^1-$topology and for $t_0$  small enough, then $d_{C^1}(H,F)$ is small  $H$ is accessible. 




	




Let  $\tau(p)$ be the period of $p$. If  $F^{\tau(p)}|\mathcal{F}^c(p)$ is an irrational rotation, then $H$ satifies the properties that we are looking for. If not, perturb a little to obtain irrational rotation. An small perturbation does not destroy the  Morse-Smale dynamics on $\mathcal{F}^c(0)$.

\bibliography{ourbib-BFT2016}{}
\bibliographystyle{plain}
\end{document}